\documentclass{conm-p-l}

\newcommand{\qform}[1]{{\left\langle{#1}\right\rangle}}
\newcommand{\pform}[1]{{\langle\!\langle{#1}\rangle\!\rangle}}

\DeclareMathOperator{\Pf}{Pf}
\DeclareMathOperator{\Nrd}{Nrd}
\DeclareMathOperator{\Sym}{Sym}
\DeclareMathOperator{\Skew}{Skew}
\DeclareMathOperator{\Int}{Int}

\newcommand{\F}{\mathbb{F}}
\newcommand{\Id}{\operatorname{Id}}

\newtheorem{prop}{Proposition}[section]
\newtheorem{lem}[prop]{Lemma}
\newtheorem{thm}[prop]{Theorem}
\newtheorem{cor}[prop]{Corollary}
\newtheorem*{BRV}{Proposition}

\theoremstyle{remark}
\newtheorem{rem}[prop]{Remark}

\title{On the Pfister Number of Quadratic Forms}
\author{R. Parimala}
\address{Department of Mathematics and Computer Science \\
Emory University \\
400 Dowman Drive \\
Atlanta, Georgia 30322  USA}
\email{parimala@mathcs.emory.edu}

\author{V. Suresh}
\address{Department of Mathematics and Statistics \\
University of Hyderabad \\
Gachibowli \\
Hyderabad - 500046 India}
\email{vssm@uohyd.ernet.in}

\author{J.-P. Tignol}
\address{D\'epartement de math\'ematique\\
Universit\'e catholique de Louvain\\
chemin du cyclotron, 2\\
B-1348 Louvain-la-Neuve, Belgium}
\email{jean-pierre.tignol@uclouvain.be}
\thanks{The first author is partially supported by NSF grant DMS-0653382
and the third author is partially supported by the Fund for
    Scientific Research F.R.S.--FNRS (Belgium)}


\begin{document}

\begin{abstract}
    The generic quadratic form of even dimension $n$ with trivial
    discriminant over an arbitrary field of characteristic different
    from~$2$ containing a square root of $-1$ can be written in the Witt
    ring as a sum of $2$-fold Pfister forms using $n-2$ terms and not
    less. The number of $2$-fold Pfister forms needed to express a
    quadratic form of dimension~$6$ with trivial discriminant is
    determined in various cases.
\end{abstract}

\maketitle
\section*{Introduction}

Throughout this paper, $k$ denotes a field of characteristic different
from~$2$ in which $-1$ is a square. We use the same notation for a
quadratic form over $k$ and its Witt equivalence class in the Witt
ring $W(k)$. As usual, the quadratic form $\sum_{i=1}^na_iX^i$ with
$a_i\in k^\times$ is denoted by $\qform{a_1,\ldots,a_n}$. Since $-1$
is a square in $k$, the form $\qform{1,\ldots,1}$ is Witt equivalent
to $\qform{1}$ or $0$ according as its dimension is odd or even, hence
$W(k)$ is an algebra over the field $\F_2$ with two elements. Let
$I(k)$ be the fundamental ideal of $W(k)$, which consists of the Witt
equivalence classes of even-dimensional quadratic forms. For any
integer $m\geq1$, the $m$-th power of $I(k)$ is denoted by
$I^m(k)$. We say a quadratic form is in $I^m(k)$ if its Witt
equivalence class is in $I^m(k)$. It is well-known that for any
$m\geq1$ the ideal $I^m(k)$ is generated as a group by the classes of
$m$-fold Pfister forms, i.e., quadratic forms of the following type:
\[
\pform{a_1,\ldots,a_m}=\qform{1,a_1}\otimes\cdots\otimes
\qform{1,a_m},
\]
see \cite[Prop.~X.1.2]{L}. Brosnan,
Reichstein, and Vistoli \cite{BRV} define the \emph{$m$-Pfister
    number} $\Pf_m(q)$ of a quadratic form $q\in I^m(k)$ as the least
  number of terms in a decomposition of its 
Witt equivalence class into a sum of $m$-fold Pfister
forms. For $m$, $n\geq1$, the \emph{$(m,n)$-Pfister number}
$\Pf_k(m,n)$ is defined as the supremum of the $m$-Pfister numbers
$\Pf_m(q)$ where $q$ runs over the quadratic forms of dimension $n$
in $I^m(K)$, as $K$ varies over field extensions of $k$. In
\cite{BRV}, Pfister numbers are studied in connection with the 
essential dimension of algebraic groups.

A related invariant was defined by Parimala and Suresh in \cite{PS}
(see also Kahn's paper \cite{B}):
the \emph{length} $\lambda_m(q)$ of a quadratic form $q\in I^m(k)$ is
the least integer $r$ for which there exist $m$-fold 
Pfister forms $\pi_1$, \ldots, $\pi_r$ such that
$q\equiv\pi_1+\cdots+\pi_r\bmod I^{m+1}(k)$. In \cite{PS}, the length of 
quadratic forms was studied with reference to the $u$-invariant of
fields and some bounds were given for the length of quadratic forms
in $I^m(k)$, $1 \leq m \leq 3$. Clearly, $\lambda_m(q)\leq\Pf_m(q)$.

The following bounds were
given in \cite{BRV} for Pfister numbers of forms in $I(k)$ and
$I^2(k)$ (see also 
Proposition~\ref{prop:BRV} below):

\begin{BRV}[{\cite[Prop.~14]{BRV}}]
    $\Pf_k(1,n)\leq n$ and $\Pf_k(2,n)\leq n-2$.
\end{BRV}

Zinovy Reichstein raised the following question: Is the estimate for
the $2$-Pfister number in the proposition sharp, i.e., is $\Pf_k(2,n)=n-2$?
In this paper we answer Reichstein's question in the affirmative by
showing that the ``generic'' quadratic form $q_0$ of dimension~$n$
with trivial discriminant satisfies $\Pf_2(q_0)=n-2$, see
Theorem~\ref{thm:3} and Corollary~\ref{cor:main}. Note that for any
quadratic form $q$ of dimension~$n$ in $I^2(k)$ we have
$\lambda_2(q)\leq\frac{n-2}{2}$ (cf.\ \cite[Prop.~1.1]{B}); therefore
for the generic form $q_0$ the inequality
$\lambda_2(q_0)\leq\Pf_2(q_0)$ is strict.

The proof of Theorem~\ref{thm:3} is easily derived from a discussion
of a combinatorial analogue of Pfister numbers in
\S\ref{sec:combi}. In the last section (\S\ref{sec:low}), which is
essentially independent from \S\S\ref{sec:combi} and \ref{sec:generic},
we give some computations of Pfister numbers of quadratic forms
of dimension~$6$.

We are indebted to Zinovy Reichstein for his comments on a first
version of this paper, which helped us to improve the wording in
several points, and also to Detlev Hoffmann for suggesting an
alternative proof of Theorem~\ref{thm:3}. Ideas from this alternative
proof were used to simplify our original arguments.

\section{A combinatorial analogue}
\label{sec:combi}

Let $V$ be an arbitrary vector space over the field $\F_2$
with $2$ elements. We consider the group algebra $\F_2[V]$ as a
combinatorial analogue of the Witt ring of a field. (Indeed, the Witt
ring of any field $k$ of characteristic different from $2$ containing
a square root of $-1$ is a homomorphic image of
$\F_2[k^\times/k^{\times2}]$, see \S\ref{sec:generic}.)  
Since the addition in $V$ is multiplication in $\F_2[V]$, it
is convenient to denote by $X^v$ the image of $v\in V$ in $\F_2[V]$;
thus
\[
\F_2[V]=\Bigl\{\sum_{v\in V}\alpha_vX^v\mid\alpha_v\in\F_2
\text{ and $\{v\in V\mid \alpha_v\neq0\}$ is finite}\Bigr\},
\]
with
\[
X^0=1\qquad\text{and}\qquad X^u\cdot X^v=X^{u+v} \quad\text{for $u$,
    $v\in V$}.
\]
We consider the group homomorphisms
\[
\varepsilon_0\colon\F_2[V]\to\F_2,\qquad \varepsilon_1\colon\F_2[V]\to V
\]
defined by
\[
\varepsilon_0\Bigl(\sum_{v\in V}\alpha_vX^v\Bigr)=\sum_{v\in V}\alpha_v,\qquad
\varepsilon_1\Bigl(\sum_{v\in V}\alpha_vX^v\Bigr)=\sum_{v\in V}\alpha_vv.
\]
Thus, $\varepsilon_0$ is the augmentation map. We denote its kernel by
$I[V]$. It is an ideal since $\varepsilon_0$ is a ring homomorphism, and it is
generated as a group by elements of the form $1+X^v$ for $v\in V$,
which we call \emph{$1$-fold Pfister elements}. For $m\geq1$, the
products
\[
(1+X^{v_1})\cdots(1+X^{v_m})\in \F_2[V]
\]
with $v_1$, \ldots, $v_m\in V$ are called \emph{$m$-fold Pfister
    elements}. They span the $m$-th power of $I[V]$, which we denote
by $I^m[V]$ to mimic the Witt ring notation. Observe that $0$ is an
$m$-fold Pfister element for all $m$, since the product above is $0$
if $v_1=0$. 

For $\xi=\sum_{v\in V}\alpha_vX^v\in\F_2[V]$ we define the
\emph{support} of $\xi$ by
\[
D(\xi)=\{v\in V\mid\alpha_v=1\}\subseteq V.
\]
This notation is inspired by the usual notation for the set of
represented values of a quadratic form. (See the proof of
Theorem~\ref{thm:3} below for an example of a field $E$ such that
$W(E)$ can be identified with a group algebra $\F_2[V]$ in such a way
that the support of any $\xi\in\F_2[V]$ is the set of represented
values of the corresponding anisotropic quadratic form.)

\begin{lem}
    \label{lem:Pfelem}
    Let $\xi\in \F_2[V]$ be a nonzero element, and let $d=\lvert
    D(\xi)\rvert$ be the cardinality of the support of $\xi$.
    \begin{enumerate}
    \item[(i)]
    If $\xi\in I[V]$, then $d\geq2$ and there are $1$-fold Pfister
    elements $\pi_1$, \ldots, $\pi_p$ such that
    \[
    \xi=\pi_1+\cdots+\pi_p\qquad\text{and}\qquad p\leq d.
    \]
    If moreover $0\in D(\xi)$, the same property holds with $p\leq d-1$.
    \item[(ii)]
    If $\xi\in I[V]$ and $\varepsilon_1(\xi)=0$, then $d\geq4$ and there exist
    $2$-fold Pfister
    elements $\pi_1$, \ldots, $\pi_p$ such that
    \[
    \xi=\pi_1+\cdots+\pi_p\qquad\text{and}\qquad p\leq d-2.
    \]
    If moreover $0\in D(\xi)$, the same property holds with $p\leq d-3$.
    \end{enumerate}
\end{lem}

\begin{proof}
    (i)
    We have $d\neq0$ since $\xi\neq0$, and $d$ is even since
    $\varepsilon_0(\xi)\equiv d\bmod2$ and $\xi\in I[V]$. Therefore, we have
    \[
    \xi=\sum_{v\in D(\xi)}X^v=\sum_{v\in D(\xi)}(1+X^v),
    \]
    proving that $\xi$ is a sum of $d$ terms that are $1$-fold Pfister
    elements. If $0\in D(\xi)$, one of these terms vanishes since
    $1+X^0=0$. Thus, (i) is proved.

    (ii)
    Suppose now $\xi\in I[V]$ and $\varepsilon_1(\xi)=0$. As in case~(i), $d$ is
    even. If $d=2$, the condition
    $\varepsilon_1(\xi)=0$ yields $\xi=0$. Therefore, $d\geq4$. The other
    assertions are proved by induction on $d$. Suppose first $0\in
    D(\xi)$. Since $d\geq4$ we may find in $D(\xi)$ two distinct nonzero
    vectors $u$, $v$. Define
    \[
    \xi'=(1+X^u)(1+X^v)+\xi.
    \]
    We have $\xi'\in I[V]$ and $\varepsilon_1(\xi')=0$. Moreover,
    \[
    D(\xi')\subseteq(D(\xi)\setminus\{0,u,v\})\cup\{u+v\},
    \]
    hence $\lvert D(\xi')\rvert\leq d-2$. By induction, there exist
    $2$-fold Pfister elements $\pi_1$, \ldots, $\pi_p$ such that
    \[
    \xi'=\pi_1+\cdots+\pi_p\qquad\text{and}\qquad p\leq d-4,
    \]
    ($\xi'=0$ if $d=4$). Then
    \[
    \xi=\pi_1+\cdots+\pi_p+(1+X^u)(1+X^v)
    \]
    and the number of terms on the right side is at most $d-3$. If
    $0\notin D(\xi)$ we may still define $\xi'$ as above, and we have
    \[
    0\in D(\xi')\subseteq(D(\xi)\setminus\{u,v\})\cup\{0,u+v\},
    \]
    hence $\lvert D(\xi')\rvert\leq d$. The arguments above show that
    there exist $2$-fold Pfister elements $\pi_1$, \ldots, $\pi_p$ such
    that
    \[
    \xi'=\pi_1+\cdots+\pi_p\qquad\text{and}\qquad p\leq d-3.
    \]
    Then
    \[
    \xi=\pi_1+\cdots+\pi_p+(1+X^u)(1+X^v)
    \]
    and the number of terms on the right side is at most $d-2$.
\end{proof}

\begin{cor}
    \label{cor:I2}
    $I^2[V]=\{\xi\in I[V]\mid \varepsilon_1(\xi)=0\}$.
\end{cor}

\begin{proof}
    Lemma~\ref{lem:Pfelem}(i) shows that $I[V]$ is spanned
    by $1$-fold Pfister elements, hence $I^2[V]$ is generated as
    a group by $2$-fold Pfister elements. Since these elements lie in
    the kernel of $\varepsilon_1$, it follows that
    \[
    I^2[V]\subseteq\{\xi\in I[V]\mid \varepsilon_1(\xi)=0\}.
    \]
    The reverse inclusion readily follows from Lemma~\ref{lem:Pfelem}(ii).
\end{proof}

For $\xi\in I^m[V]$ we define the $m$-Pfister number $\Pf_m(\xi)$ as
the minimal number of
terms in a decomposition of $\xi$ as a sum of $m$-fold Pfister
elements. In particular, $\Pf_m(0)=0$ for all $m\geq1$.

\begin{prop}
    \label{prop:Pf1}
    For every $\xi\in I[V]$ we have
    \[
    \Pf_1(\xi)=
    \begin{cases}
      \lvert D(\xi)\rvert&\text{if $0\notin D(\xi)$},\\
      \lvert D(\xi)\rvert-1&\text{if $0\in D(\xi)$}.
    \end{cases}
    \]
\end{prop}

\begin{proof}
    Let $p=\Pf_1(\xi)$. Suppose $v_1$, \ldots, $v_p\in V$ are nonzero
    vectors such that
    \[
    \xi=\sum_{i=1}^p(1+X^{v_i})=p+\sum_{i=1}^p X^{v_i}.
    \]
    Then $D(\xi)\subseteq\{0,v_1,\ldots,v_p\}$, hence
    \[
    \lvert D(\xi)\rvert\leq
    \begin{cases}
      p&\text{if $0\notin D(\xi)$},\\
      p+1&\text{if $0\in D(\xi)$}.
    \end{cases}
    \]
    The reverse inequality follows from Lemma~\ref{lem:Pfelem}.
\end{proof}

We now turn to $2$-Pfister numbers.
\relax From Lemma~\ref{lem:Pfelem} it follows that for $\xi\neq0$ in
$I^2[V]$,
\begin{equation}
    \label{eq:upb}
    \Pf_2(\xi)\leq\lvert D(\xi)\rvert-2,\quad\text{and}\quad
    \Pf_2(\xi)\leq\lvert D(\xi)\rvert-3\text{ if $0\in D(\xi)$}.
\end{equation}

In the rest of this section, we explicitly construct elements for
which the upper bound is reached. The following general
observation is crucial for the proof: every linear map $\varphi\colon
V\to W$ between $\F_2$-vector spaces induces a ring homomorphism
$\varphi_*\colon\F_2[V]\to\F_2[W]$ by
\begin{equation}
  \label{eq:phi}
  \varphi_*\Bigl(\sum_{v\in V}\alpha_vX^v\Bigr)= \sum_{v\in V}
  \alpha_vX^{\varphi(v)}. 
\end{equation}
The homomorphism $\varphi_*$ maps $1$-fold Pfister elements in
$\F_2[V]$ to (possibly zero) $1$-fold Pfister elements in $\F_2[W]$,
hence also $m$-fold Pfister elements in $\F_2[V]$ to $m$-fold Pfister
elements in $\F_2[W]$, for every $m\geq1$. Consequently, for every
$\xi\in I^m[V]$ we have $\varphi_*(\xi)\in I^m[W]$ and
\begin{equation*}
  \Pf_m\bigl(\varphi_*(\xi)\bigr)\leq\Pf_m(\xi).
\end{equation*}

Now, let $V$ be an $\F_2$-vector space of finite dimension $n>1$, and let
$e=(e_i)_{i=1}^n$ be a base of $V$. We define $e_0=\sum_{i=1}^ne_i$
and
\[
\xi_e=n+1+\sum_{i=0}^nX^{e_i}\in\F_2[V].
\]
It is readily verified that $\varepsilon_0(\xi)=0$ and
$\varepsilon_1(\xi)=0$, so
$\xi_e\in I^2[V]$, and the support of $\xi_e$ is
\[
D(\xi_e)=
\begin{cases}
    \{0,e_0,e_1,\ldots,e_n\}&\text{if $n$ is even},\\
    \{e_0,e_1,\ldots,e_n\}&\text{if $n$ is odd}.
\end{cases}
\]
Therefore, \eqref{eq:upb} yields the same inequality when $n$ is odd
or even:
\begin{equation}
  \label{eq:upb1}
  \Pf_2(\xi_e)\leq n-1.
\end{equation}
The following proposition shows that $\Pf_2(\xi_e)$ reaches the bound
in \eqref{eq:upb}.

\begin{prop}
    \label{prop:Pf2}
    $\Pf_2(\xi_e)=n-1$.
\end{prop}

\begin{proof}
    We use induction on $n$. If $n=2$, we have
    \[
    \xi_e=(1+X^{e_1})(1+X^{e_2}),
    \]
    so $\Pf_2(\xi_e)=1$. If $n=3$, then
    \[
    \xi_e=X^{e_1}+X^{e_2}+X^{e_3}+X^{e_1}X^{e_2}X^{e_3}.
    \]
    This element is not a $2$-fold Pfister element since $0\notin
    D(\xi_e)$, hence $\Pf_2(\xi_e)>1$. On the other hand,
    $\Pf_2(\xi_e)\leq2$ by \eqref{eq:upb1}, hence $\Pf_2(\xi_e)=2$.

    For the rest of the proof, suppose $n>3$. Let $p=\Pf_2(\xi_e)$ and let
    $\pi_1$, \ldots, $\pi_p$ be $2$-fold Pfister elements such that
    \begin{equation}
      \label{eq:xi}
      \xi_e=\pi_1+\cdots+\pi_p.
    \end{equation}
    We have $D(\xi_e)\subseteq\bigcup_{i=1}^pD(\pi_i)$, hence $e_n\in
    D(\pi_i)$ for some $i=1$, \ldots, $p$. Renumbering, we may assume
    $e_n\in D(\pi_p)$, hence
    \begin{equation}
    \label{eq:pip}
    \pi_p=1+X^{e_n}+X^v+X^{e_n+v}\qquad\text{for some $v\in V$}.
    \end{equation}
    Let $W\subseteq V$ be the $\F_2$-span of $e_1$, \ldots, $e_{n-1}$,
    and let $f_0=\sum_{i=1}^{n-1}e_i\in W$. Clearly,
    $f=(e_i)_{i=1}^{n-1}$ is a base of $W$, and the element
    $\xi_{f}\in\F_2[W]$ built on the same model as $\xi_e$ is
    \[
    \xi_f=n+X^{f_0}+\sum_{i=1}^{n-1}X^{e_i}.
    \]
    Consider the linear map $\varphi\colon V\to W$ defined by
    \[
    \varphi(e_i)=
    \begin{cases}
      e_i&\text{for $i=1$, \ldots, $n-1$},\\
      0&\text{for $i=n$}.
    \end{cases}
    \]
    The ring homomorphism $\varphi_*\colon\F_2[V]\to\F_2[W]$ induced
    by $\varphi$ as in \eqref{eq:phi} above satisfies
    $\varphi_*(X^{e_n})=1$. Since $\varphi(e_0)=f_0$, it follows that
    $\varphi_*(\xi_e)=\xi_{f}$, hence \eqref{eq:xi} yields
    \[
    \xi_f=\varphi_*(\pi_1)+\cdots+\varphi_*(\pi_p).
    \]
    In view of \eqref{eq:pip}, we have $\varphi_*(\pi_p)=0$, hence the
    preceding equation shows that $\Pf_2(\xi_f)\leq p-1$. Since $\dim
    W=n-1$, the induction hypothesis yields $\Pf_2(\xi_f)=n-2$,
    hence $n-1\leq p$. The reverse inequality holds by \eqref{eq:upb1},
    hence the proposition is proved.
\end{proof}

\section{Pfister numbers of generic forms}
\label{sec:generic}

Let $k$ be an arbitrary field of characteristic different from $2$
containing a square root of $-1$, and let $V_k=k^\times/k^{\times2}$
be the group of square classes in $k$, which we view as an
$\F_2$-vector space. The map
\[
\Psi\colon V_k\to W(k)
\]
defined by $\Psi(a\,k^{\times2})=\qform{a}$ for $a\in k^\times$ is
multiplicative, hence it
induces a surjective $\F_2$-algebra homomorphism
\[
\Psi_*\colon\F_2[V_k]\to W(k).
\]
The map $\Psi_*$ carries $1$-fold Pfister elements in $\F_2[V_k]$ to
$1$-fold Pfister forms in $W(k)$, hence also $m$-fold Pfister elements
to $m$-fold Pfister forms for all $m\geq1$. Therefore,
$\Psi_*(I^m[V_k])=I^m(k)$ and we have
\begin{equation}
  \label{eq:PfPsi}
  \Pf_m\bigl(\Psi_*(\xi)\bigr)\leq\Pf_m(\xi) \qquad\text{for all
    $\xi\in I^m[V_k]$}.
\end{equation}
We may then use Lemma~\ref{lem:Pfelem} to give a short proof of
Proposition~14 of \cite{BRV}, including a minor refinement:

\begin{prop}
    \label{prop:BRV}
    Let $q$ be a quadratic form of dimension $n$ over a field $k$
    containing a square root of $-1$.
    \begin{enumerate}
    \item[(i)]
    If $q\in I(k)$, then $\Pf_1(q)\leq n$.
    If moreover $q$ represents $1$, then $\Pf_1(q)\leq n-1$.
    \item[(ii)]
    If $q\in I^2(k)$, then $\Pf_2(q)\leq
    n-2$. If moreover $q$ represents $1$, then $\Pf_2(q)\leq n-3$.
    \end{enumerate}
\end{prop}

\begin{proof}
    Let $q=\qform{a_1,\ldots, a_n}$. Consider then
    \[
    \xi=(a_1\,k^{\times2})+\cdots+(a_n\,k^{\times2})\in\F_2[V_k].
    \]
    We have $\Psi_*(\xi)=q$ and $D(\xi)=\{a_1\,k^{\times2},\ldots,
    a_n\,k^{\times2}\}$, so $\lvert D(\xi)\rvert\leq n$. If $q\in I(k)$,
    then $n$ is even hence $\xi\in I[V_k]$. Lemma~\ref{lem:Pfelem}(i)
    then yields $\Pf_1(\xi)\leq n$, and by \eqref{eq:PfPsi} it follows
    that $\Pf_1(q)\leq n$. If $q$ represents~$1$, then we may assume
    $a_1=1$, hence $D(\xi)$ contains the zero element of
    $V_k$. Lemma~\ref{lem:Pfelem}(i) then yields $\Pf_1(\xi)\leq n-1$,
    and by \eqref{eq:PfPsi} it follows that $\Pf_1(q)\leq n-1$.

    If $q\in I^2(k)$, then $a_1\ldots a_n\in k^{\times2}$ hence
    $\varepsilon_1(\xi)=0$. By Corollary~\ref{cor:I2} we have $\xi\in
    I^2[V_k]$, and Lemma~\ref{lem:Pfelem}(ii) yields $\Pf_2(\xi)\leq
    n-2$. Therefore, by \eqref{eq:PfPsi} we get $\Pf_2(q)\leq
    n-2$. Again, if $q$ represents~$1$ we may assume $0\in D(\xi)$,
    and the preceding inequalities can be strengthened to
    \[
    \Pf_2(q)\leq\Pf_2(\xi)\leq n-3.
    \]
\end{proof}

For the rest of this section, fix an arbitrary integer
$n\geq2$. Consider $n$ independent indeterminates $x_1$, \ldots, $x_n$
over $k$ and let 
\[
x_0=x_1\cdots x_n.
\]
Over the field $K=k(x_1,\ldots,x_n)$, we consider the following
quadratic forms:
\begin{align*}
    q & =\qform{x_1,\ldots,x_n},& q_0& =\qform{x_0,x_1,\ldots, x_n},\\
    q'& =\qform{1,x_1,\ldots, x_n},& q'_0& =\qform{1,x_0,x_1,\ldots, x_n}.
\end{align*}
If $n$ is even, then $q\in I(K)$ and $q'_0\in I^2(K)$. If $n$ is odd,
then $q'\in I(K)$ and $q_0\in I^2(K)$.

\begin{thm}
    \label{thm:3}
    If $n$ is even, then
    \[
    \Pf_1(q)=n\qquad\text{and}\qquad\Pf_2(q'_0)=n-1.
    \]
    If $n$ is odd, then
    \[
    \Pf_1(q')=n\qquad\text{and}\qquad\Pf_2(q_0)=n-1.
    \]
\end{thm}

\begin{proof}
    Let $k_{\text{alg}}$ be an algebraic closure of $k$. Embed $k$ in
    the field of iterated Laurent series
    $E=k_{\text{alg}}((x_1))\cdots((x_n))$. Applying Springer's theorem
    in \cite[Cor.~VI.1.7]{L} recursively, we obtain a ring isomorphism
    \begin{equation*}
    \Theta\colon\,W(E)\stackrel{\sim}{\to}
    \F_2[(\mathbb{Z}/2\mathbb{Z})^n],
    \end{equation*}
    which maps
    $W(k_{\text{alg}})$ onto $\F_2$ and maps
    the quadratic form $\qform{x_i}$ to $X^{e_i}$, where
    $e_i$ is the $i$-th element in the standard base of
    $(\mathbb{Z}/2\mathbb{Z})^n$ as an $\F_2$-vector space, for $i=1$,
    \ldots, $n$. Note that the $(x_1,\ldots,x_n)$-adic valuation on
    $E$ yields an isomorphism
    \[
    V_E=E^{\times}/E^{\times2}\simeq(\mathbb{Z}/2\mathbb{Z})^n
    \]
    which maps $x_i\,E^{\times2}$ to $e_i$ for $i=1$, \ldots,
    $n$. Using this isomorphism as an identification, we may view
    $\Theta$ as the inverse map of $\Psi_*\colon\F_2[V_E]\to W(E)$,
    which is thus an isomorphism in this case.

    Letting $e_0=\sum_{i=1}^ne_i$, we have
    \begin{align*}
      \Theta(q_E) & =\sum_{i=1}^n X^{e_i},& \Theta(q_{0E}) &
      =\sum_{i=0}^nX^{e_i},\\
      \Theta(q'_E) & =1+\sum_{i=1}^n X^{e_i},& \Theta(q'_{0E}) &
      =1+\sum_{i=0}^nX^{e_i},
    \end{align*}
    hence in the notation of \S\ref{sec:combi} with
    $V=(\mathbb{Z}/2\mathbb{Z})^n$, we have
    \[
    \xi_e=
    \begin{cases}
      \Theta(q_{0E})&\text{if $n$ is odd},\\
      \Theta(q'_{0E})&\text{if $n$ is even}.
    \end{cases}
    \]
    The isomorphism $\Theta$ maps $m$-fold Pfister forms in $W(E)$ to
    $m$-fold Pfister elements in $\F_2[V]$, hence it preserves
    $m$-Pfister numbers. Therefore, Proposition~\ref{prop:Pf2} yields
    \[
    \Pf_2(q_{0E})=n-1\quad\text{if $n$ is odd}\qquad\text{and}\qquad
    \Pf_2(q'_{0E})=n-1\quad\text{if $n$ is even}.
    \]
    Similarly, Proposition~\ref{prop:Pf1} yields
    \[
    \Pf_1(q_{E})=n\quad\text{if $n$ is even}\qquad\text{and}\qquad
    \Pf_1(q'_{E})=n\quad\text{if $n$ is odd}.
    \]
    Since scalar extension preserves $m$-Pfister forms, it follows that
    \[
    \Pf_1(q)\geq n\quad\text{and}\quad\Pf_2(q'_0)\geq n-1\qquad\text{if
      $n$ is even},
    \]
    \[
    \Pf_1(q')\geq n\quad\text{and}\quad\Pf_2(q_0)\geq n-1\qquad\text{if
      $n$ is odd}.
    \]
    The reverse inequalities follow from Proposition~\ref{prop:BRV}.
\end{proof}

\begin{cor}
    \label{cor:main}
    $\Pf_k(1,m)=m$ for any even integer $m\geq2$ and $\Pf_k(2,m)=m-2$
    for any even integer $m\geq4$.
\end{cor}

\begin{proof}
    For $m$ even, $m\geq2$, the form $q$ above (with $n=m$) has
    dimension $m$ and satisfies
    $q\in I(F)$ and $\Pf_1(q)=m$, so $\Pf_k(1,m)\geq m$. Similarly, for
    $m$ even, $m\geq4$, the form $q_0$ above (with $n=m-1$) has
    dimension $m$ and satisfies $q_0\in I^2(F)$ and $\Pf_2(q_0)=m-2$, so
    $\Pf_k(2,m)\geq m-2$. The reverse inequalities follow from
    \cite[Prop.~14]{BRV} (see the Introduction or
    Proposition~\ref{prop:BRV}).
\end{proof}

\begin{rem}
    \label{rem}
    A form with the same $2$-Pfister number as $q'_0$ can be obtained by
    scaling $q_0$: we have
    \[
    \qform{x_1}q_0=\qform{1,x_1x_2,\ldots, x_1x_n,x_0x_1}
    \]
    and $x_1x_2$, \ldots, $x_1x_n$ may be regarded as independent
    indeterminates. If $n$ is odd we have
    \[
    x_0x_1\equiv(x_1x_2)\cdots(x_1x_n)\bmod K^{\times2},
    \]
    hence $\qform{x_1}q_0$ is isometric to a quadratic form like $q'_0$
    in the indeterminates $x_1x_2$, \ldots, $x_1x_n$. Embedding $K$ in
    $E$ as in the proof of Theorem~\ref{thm:3}, we obtain
    $\Pf_2(\qform{x_1}q_0)=n-2$. Details are left to the reader.
\end{rem}

\section{Low-dimensional forms}
\label{sec:low}

Let $k$ be an arbitrary field of characteristic different from~$2$
containing a square root of $-1$. In this section, we obtain some
information on the $2$-Pfister number of quadratic forms of
dimension~$4$ or $6$ over $k$.
\medbreak\par
The case of anisotropic quadratic forms $q\in I^2(k)$ of dimension~$4$
is clear: if $q$ represents~$1$, then $q$ is a $2$-fold Pfister form,
so $\Pf_2(q)=1$. On the other hand, if $q$ does not represent~$1$,
then $\Pf_2(q)>1$ and Proposition~\ref{prop:BRV} yields $\Pf_2(q)=2$.
\medbreak\par
We next consider anisotropic forms of dimension~$6$ in $I^2(k)$. Of
course, $\Pf_2(q)>1$ for any such form $q$. If $q$ represents~$1$, it
follows from Proposition~\ref{prop:BRV} that $\Pf_2(q)=2$ or $3$. The
Stiefel-Whitney invariant $w_4(q)\in H^4(k,\mu_2)$ discriminates
between the two cases, as the next proposition shows. (See
\cite[\S3]{Mi} or \cite[\S17]{GMS}
for a discussion of Stiefel-Whitney invariants of quadratic forms.)

\begin{prop}
    \label{prop:sw}
    Let $q$ be an anisotropic quadratic form of dimension~$6$. Assume
    $q\in I^2(k)$ and $q$ represents~$1$. If $w_4(q)=0$, then
    $\Pf_2(q)=2$. If $w_4(q)\neq0$, then $\Pf_2(q)=3$.
\end{prop}

\begin{proof}
    In view of Proposition~\ref{prop:BRV}, it suffices to show that
    $\Pf_2(q)=2$ if and only if $w_4(q)=0$. Assume first $\Pf_2(q)=2$
    so that
    \[
    q=\qform{x_1,x_2,x_1x_2,y_1,y_2,y_1y_2}\qquad\text{for some $x_1$,
      \ldots, $y_4\in k^\times$}.
    \]
    For $x\in k^\times$, denote by $(x)\in H^1(k,\mu_2)$ the cohomology
    class associated to the square class of $x$. An explicit
    computation yields
    \[
    w_4(q)=(x_1)\cup(x_2)\cup(y_1)\cup(y_2).
    \]
    Since $q$ represents~$1$, the form $\qform{1}\perp q$ is isotropic.
    The $4$-fold Pfister form $\pform{x_1,x_2,y_1,y_2}$ which contains
    $\qform{1}\perp q$ as a subform is hyperbolic.
    Therefore,
    $(x_1)\cup(x_2)\cup(y_1)\cup(y_2)=0$ by \cite[Satz~1.6]{Ar}.

    For the converse, let
    \[
    q=\qform{1,a,b,c,d,abcd}\qquad\text{for some $a$, $b$, $c$, $d\in
      k^\times$}.
    \]
    Then
    \[
    w_4(q)=(a)\cup(b)\cup(c)\cup(d).
    \]
    Since $w_4(q)=0$ by hypothesis, Theorem~1 of \cite{AEJ} shows that
    the $4$-fold Pfister
    form $\pform{a,b,c,d}$ is hyperbolic. It follows that the
    $9$-dimensional subform $q\perp\qform{ab,ac,ad}$ is isotropic, hence
    $q$ represents a nonzero element of the form $a(bx^2+cy^2+dz^2)$ for
    some $x$, $y$, $z\in k$. Let $b'=bx^2+cy^2+dz^2\in k^\times$. Since
    the form $\qform{b,c,d}$ represents $b'$, we may find $c'$, $d'\in
    k^\times$ such that
    \[
    \qform{b,c,d}=\qform{b',c',d'}.
    \]
    Comparing discriminants, we have $bcd\equiv b'c'd'\bmod
    k^{\times2}$, hence
    \[
    q=\qform{1,a,b',c',d',ab'c'd'}.
    \]
    The form $\qform{1,a,b'}$ is anisotropic since $q$ is anisotropic,
    hence the $2$-fold Pfister form $\pform{a,b'}$ is anisotropic. On
    the other hand, the form $q\perp\qform{ab'}$ is isotropic since $q$
    represents $ab'$, hence $\pform{a,b'}$ represents a nonzero element
    of the form $c'r^2+d's^2+ab'c'd't^2$, for some $r$, $s$, $t\in
    k$. Let $c''=c'r^2+d's^2+ab'c'd't^2\in k^\times$, and let $d''\in
    k^\times$ be such that
    \[
    \qform{c',d',ab'c'd'}=\qform{c'',d'',ab'c''d''}.
    \]
    Thus
    \[
    q=\qform{1,a,b',c'', d'',ab'c''d''}.
    \]
    Since $\pform{a,b'}$ represents $c''$, the $3$-fold Pfister
    form $\pform{a,b',c''}$ is hyperbolic, and therefore its
    $5$-dimensional subform $\qform{1,a,b',c'',ab'c''}$ is
    isotropic. Thus, $\qform{1,a,b',c''}$ represents $ab'c''$, and we
    may find $u$, $v\in k^\times$ such that
    \[
    \qform{1,a,b',c''}=\qform{ab'c'',u,v,uv}.
    \]
    Thus
    \[
    q=\pform{u,v}+\pform{d'',ab'c''},
    \]
    and $\Pf_2(q)=2$.
\end{proof}

For arbitrary $6$-dimensional anisotropic quadratic forms in $I^2(k)$,
the $2$-Pfister number is $2$, $3$ or $4$. Note that scaling has an
important effect on the Pfister number although it does not change
the Stiefel-Whitney class. Indeed, by Theorem~\ref{thm:3} and
Remark~\ref{rem}, if $x_1$, \ldots, $x_5$ are independent
indeterminates and
\[
q=\qform{x_1,x_2,x_3,x_4,x_5, x_1x_2x_3x_4x_5},
\]
then
\[
\Pf_2(q)=4\qquad\text{and}\qquad\Pf_2(\qform{x_1}q)=3.
\]
On the other hand,
\[
\qform{x_1x_2x_3}q=\pform{x_1x_2,x_1x_3}+
\pform{x_4x_5,x_1x_2x_3x_4},
\]
hence
\[
\Pf_2(\qform{x_1x_2x_3}q)=2.
\]
More generally, the same computation shows that for an arbitrary
anisotropic form $q\in I^2(k)$ of dimension~$6$, if $d\in k^\times$ is
the discriminant of some $3$-dimensional subform of $q$, then
$\Pf_2(\qform{d}q)=2$. In the rest of this section, we give
necessary and sufficient conditions on $q$ for $\Pf_2(q)\leq3$
as well as for $\Pf_2(q)= 2$.

\medbreak
\par
As seen before, every quadratic form of dimension $6$ in $I^2(k)$ is
a scalar multiple of a form $q$ with $\Pf_2(q)= 2$.
Fix a decomposition
\[
q=\pform{a,b}+\pform{c,d}=\qform{a,b,ab,c,d,cd}.
\]
To this decomposition is associated the biquaternion algebra
$D=(a,b)_k\otimes(c,d)_k$, which is Brauer-equivalent to the Clifford
algebra of $q$, and the orthogonal involution $\sigma$ on $D$ that is
the tensor product of the conjugation involutions on $(a,b)_k$ and
$(c,d)_k$. The algebra $D$ is division since $q$ is anisotropic, see
\cite[(16.5)]{KMRT}.

\begin{thm}
    \label{thm:1}
    For $\lambda\in k^\times$, we have $\Pf_2(\qform{\lambda}q)=2$ if
    and only if $\lambda^2$ is the reduced norm
    of some $\sigma$-symmetric element in $D$, i.e.,
    \[
    \lambda^2=\Nrd_D(u) \qquad\text{for some $u\in\Sym(D,\sigma)$}.
    \]
\end{thm}

\begin{proof}
    Let $(a,b)_k^0$ (resp.\ $(c,d)_k^0$) be the $k$-vector space of pure
    quaternions in $(a,b)_k$ (resp.\ in $(c,d)_k$). The vector space of
    $\sigma$-skew-symmetric elements in $D$ is
    \[
    \Skew(D,\sigma)=\bigl((a,b)_k^0\otimes1\bigr)\oplus
    \bigl(1\otimes(c,d)_k^0\bigr).
    \]
    Let $p_\sigma$ be the linear operator on $\Skew(D,\sigma)$ defined
    by
    \[
    p_\sigma(x\otimes1+1\otimes y)=x\otimes1-1\otimes y \qquad\text{for
      $x\in(a,b)_k^0$ and $y\in(c,d)_k^0$}.
    \]
    The formula $q_\sigma(s)=s\,p_\sigma(s)$ defines a quadratic form on
    $\Skew(D,\sigma)$, and we have
    \[
    q_\sigma\simeq\qform{a,b,ab,c,d,cd}=q.
    \]
    Suppose now $\Pf_2(\qform{\lambda}q)=2$. We fix a decomposition
    \[
    \qform{\lambda}q=\pform{a',b'}+\pform{c',d'}=
    \qform{a',b',a'b',c',d',c'd'}.
    \]
    The Clifford algebras of $q$ and $\qform{\lambda}q$ are isomorphic,
    hence we may identify
    \[
    D=(a',b')_k\otimes(c',d')_k.
    \]
    Let $\sigma'$ be the orthogonal involution on $D$ that is the tensor
    product of the conjugation involutions on $(a',b')_k$ and
    $(c',d')_k$. By \cite[(2.7)]{KMRT}, there is a unit $u\in\Sym(D,\sigma)$
    such that $\sigma'=\Int(u)\circ\sigma$, i.e.,
    $\sigma'(x)=u\sigma(x)u^{-1}$ for all $x\in D$. On
    $\Skew(D,\sigma')$ we may define a linear operator $p_{\sigma'}$ and
    a quadratic form $q_{\sigma'}$ in the same way as $p_\sigma$ and
    $q_\sigma$ were defined on $\Skew(D,\sigma)$, and we have
    \[
    q_{\sigma'}\simeq\qform{\lambda}q.
    \]
    It is easily seen that
    $\Skew(D,\sigma')=u\Skew(D,\sigma)=\Skew(D,\sigma)u^{-1}$. The
    linear operator $p'$ on $\Skew(D,\sigma')$ defined by
    \[
    p'(s')=up_\sigma(s'u) \qquad\text{for $s'\in\Skew(D,\sigma')$}
    \]
    satisfies
    \begin{equation}
      \label{eq:s}
      s'p'(s')=s'up_\sigma(s'u)=q_\sigma(s'u)\in k.
    \end{equation}
    Therefore, by \cite[(16.22)]{KMRT}, the map $p'$ is a multiple of
    $p_{\sigma'}$: there exists $\lambda_1\in k^\times$ such that
    $p'=\lambda_1p_{\sigma'}$. It follows that
    \begin{equation}
      \label{eq:lambda}
      s'p'(s')=\lambda_1q_{\sigma'}(s') \qquad\text{for
        $s'\in\Skew(D,\sigma')$},
    \end{equation}
    and \eqref{eq:s} shows that the map $s'\mapsto s'u$ is an isometry
    $\qform{\lambda_1}q_{\sigma'}\simeq q_\sigma$. Hence
    $\qform{\lambda_1\lambda}q\simeq q$ and
    $\lambda\lambda_1^{-1}$ is the multiplier of a similitude
    of $q$. By \cite[(15.34)]{KMRT}, we may find $\lambda_2\in k^\times$
    and $v\in D^\times$ such that
    \begin{equation}
      \label{eq:new}
      \lambda\lambda_1^{-1}=\lambda_2^2\Nrd_D(v).
    \end{equation}
    On the other hand, for $s'\in\Skew(D,\sigma')$ we have
    $q_\sigma(s'u)^2=\Nrd_D(s'u)$ and $q_{\sigma'}(s')^2=\Nrd_D(s')$ by
    \cite[(16.25)]{KMRT}, hence \eqref{eq:s} and \eqref{eq:lambda} yield
    \[
    \lambda_1^2=\Nrd_D(u).
    \]
    Using this equation, we derive from \eqref{eq:new}:
    \[
    \lambda^2=\lambda_1^2\lambda_2^4\Nrd_D(v)^2=
    \Nrd_D\bigl(\lambda_2vu\sigma(v)\bigr).
    \]
    Since $\lambda_2vu\sigma(v)\in\Sym(D,\sigma)$, the element $\lambda$
    satisfies the condition in the theorem.

    Conversely, assume $\lambda^2=\Nrd_D(u)$ for some
    $u\in\Sym(D,\sigma)$. Define an
    orthogonal involution $\sigma'$ on $D$ by
    $\sigma'=\Int(u)\circ\sigma$. By \cite[(7.3)]{KMRT}, the discriminant of
    $\sigma'$ is $\Nrd_D(u)=\lambda^2$, hence by \cite[(15.12)]{KMRT} we may
    find quaternion subalgebras $(a',b')_k$, $(c',d')_k\subseteq D$ such
    that
    \[
    D=(a',b')_k\otimes(c',d')_k,
    \]
    and $\sigma'$ is the tensor product of the conjugations on
    $(a',b')_k$ and $(c',d')_k$. We may then define $p_{\sigma'}$ and
    $q_{\sigma'}$ as above, and we have
    \begin{equation}
      \label{eq:q}
      q_{\sigma'}\simeq\qform{a',b',a'b',c',d',c'd'}.
    \end{equation}
    On the other hand, define a linear operator $p_0$ and a quadratic
    form $q_0$ on $\Skew(D,\sigma')$ by
    \[
    p_0(s')=\lambda^{-1}up_\sigma(s'u),\qquad
    q_0(s')=\lambda^{-1}q_\sigma(s'u) \quad\text{for
      $s'\in\Skew(D,\sigma')$}.
    \]
    By definition, we have
    \begin{equation}
      \label{eq:q0}
      q_0\simeq\qform{\lambda}q_\sigma\simeq\qform{\lambda}q.
    \end{equation}
    Moreover, $s'p_0(s')=q_0(s')\in k$ for $s'\in\Skew(D,\sigma')$,
    hence $p_0$ is a multiple of $p_{\sigma'}$ by \cite[(16.22)]{KMRT}: we have
    $p_0=\mu p_{\sigma'}$ for some $\mu\in k^\times$, hence also
    $q_0=\mu q_{\sigma'}$. For $s'\in\Skew(D,\sigma')$ we have by
    \cite[(16.25)]{KMRT}
    \[
    p_0^2(s')=\lambda^{-2}up_\sigma(up_\sigma(s'u)u) =
    \lambda^{-2}\Nrd_D(u) p_\sigma^2(s'u)u^{-1}.
    \]
    Since $p_\sigma^2=\Id$ and $\Nrd_D(u)=\lambda^2$, it follows that
    $p_0^2=\Id$. Now, we also have $p_{\sigma'}^2=\Id$, hence
    $\mu=\pm1$. Therefore, $q_0\simeq\qform{\pm1}q_{\sigma'}\simeq
    q_{\sigma'}$. By \eqref{eq:q} and \eqref{eq:q0} we have
    \[
    \qform{\lambda}q\simeq\qform{a',b',a'b',c',d',c'd'},
    \]
    hence $\Pf_2(\qform{\lambda}q)=2$.
\end{proof}

Note that the group $S(q)\subset k^\times$ of spinor norms of $q$ can
be described in terms of $D$: we have by \cite[(15.34)]{KMRT}
\[
S(q)=\{\lambda\in k^\times\mid\lambda^2\in\Nrd_D(D^\times)\}.
\]
Therefore, the following is a direct consequence of
Theorem~\ref{thm:1}:

\begin{cor}
    Let $q\in I^2(k)$ be an anisotropic quadratic form of dimension~$6$,
    and let $\lambda\in k^\times$. If
    $\Pf_2(q)=\Pf_2(\qform{\lambda}q)=2$, then $\lambda$ is a spinor
    norm of $q$.
\end{cor}

We now turn to a characterization of quadratic forms of dimension~$6$
with $2$-Pfister number at most~$3$.

\begin{thm}
    Let $q\in I^2(k)$ be an anisotropic quadratic form of
    dimension~$6$. We have $\Pf_2(q)\leq3$ if and only if there exist a
    $4$-dimensional quadratic form $q_1$ over $k$ and scalars $\mu$,
    $\mu'$, $\nu\in k^\times$ satisfying the following conditions:
    \begin{enumerate}
    \item[(i)]
    $q\simeq q_1\perp\qform{\mu'}\pform{\nu}$;
    \item[(ii)]
    $\Pf_2(q_1\perp\qform{\mu}\pform{\nu})\leq2$;
    \item[(iii)]
    $\pform{\mu,\mu',\nu}=0$.
    \end{enumerate}
\end{thm}

\begin{proof}
    Suppose $\Pf_2(q)\leq3$, and let
    \[
    q=\pform{x_1,y_1}+\pform{x_2,y_2}+\pform{x_3,y_3}=
    \qform{x_1,y_1,x_1y_1,x_2,y_2,x_2y_2} + \qform{1,x_3,y_3,x_3y_3}.
    \]
    Since the dimension of $q$ is $6$, there exists a $2$-dimensional
    form $\qform{\mu}\pform{\nu}$ that is a subform of
    $\qform{x_1,y_1,x_1y_1,x_2,y_2,x_2y_2}$ and of
    $\qform{1,x_3,y_3,x_3y_3}$. Thus, we can write
    \begin{equation}
      \label{eq:4}
      \qform{x_1,y_1,x_1y_1,x_2,y_2,x_2y_2}
      =q_1\perp\qform{\mu}\pform{\nu}
    \end{equation}
    and
    \begin{equation}
      \label{eq:5}
      \qform{1,x_3,y_3,x_3y_3}  = \qform{\mu_1,\mu_2}\perp\qform{\mu}\pform{\nu}
    \end{equation}
    for some $4$-dimensional quadratic form $q_1$ and some scalars
    $\mu_1$, $\mu_2$. Equation~\eqref{eq:4} readily
    yields~(ii). Comparing discriminants on each side of \eqref{eq:5},
    we see that
    \[
    \qform{\mu_1,\mu_2}=\qform{\mu'}\pform{\nu}\qquad\text{for some $\mu'\in
      k^\times$}.
    \]
    Therefore, adding \eqref{eq:4} and \eqref{eq:5} yields~(i). Finally,
    \eqref{eq:5} shows that $\qform{\mu,\mu'}\pform{\nu}$
    represents~$1$, hence $\qform{1,\mu,\mu',\mu\nu,\mu'\nu}$ is
    isotropic. Since this form is contained in the $3$-fold Pfister form
    $\pform{\mu,\mu',\nu}$, we have~(iii).

    Conversely, suppose (i), (ii), and (iii) hold for some
    $4$-dimensional quadratic form $q_1$ and some scalars $\mu$, $\mu'$,
    $\nu\in k^\times$. Since
    \[
    \pform{\mu,\mu',\nu}=\qform{\mu,\mu\nu,\mu',\mu'\nu}\perp
    \qform{1,\nu,\mu\mu',\mu\mu'\nu},
    \]
    condition~(iii) yields
    \[
    \qform{\mu}\pform{\nu}\perp\qform{\mu'}\pform{\nu}=\pform{\nu,\mu\mu'}.
    \]
    Therefore, we derive from~(i) that
    \[
    q=\bigl(q_1\perp\qform{\mu}\pform{\nu}\bigr)+\pform{\nu,\mu\mu'}.
    \]
    Since $\Pf_2(q_1\perp\mu\pform{\nu})\leq2$ by~(ii), it follows that
    $\Pf_2(q)\leq3$.
\end{proof}

\end{document}